\documentclass[12pt,a4paper]{article}

\usepackage{hyperref}
\usepackage[english]{babel} 
\usepackage{amsthm}
\usepackage{amsmath}
\usepackage{amsfonts}
\usepackage{amssymb}
\usepackage{enumerate}
\usepackage{amscd}
\usepackage{graphicx}
\usepackage[latin1]{inputenc}  
\usepackage{makeidx}
\usepackage[all]{xy}

\makeindex

\newcounter{Def}[section]

\theoremstyle{plain}
\newtheorem{Proposition}[Def]{Proposition}

\newtheorem{Remark}[Def]{Remark}

\newtheorem{Lemma}[Def]{Lemma}
\newtheorem{Theorem}[Def] {Theorem}
\newtheorem{Kor}[Def] {Corollary}

\newtheorem{Definition}[Def]{Definition}
\newtheorem{Def/Not}[Def]{Definition/Notation}

\def\Aut              {{\rm Aut}}

\def\ZZ          {\mathbb Z}
\def\calc          {{\mathcal C}}

\def\KK           {{\mathbb K}}
\def\id               {{\rm id}}
\def\id                 {\mathrm{id}}
\def\Hom{\textnormal{Hom}}
\def\End            {\text{End}}
\def\mod            {\text{-}\mathrm{mod}}
\def\Inn            {\mathrm{Inn}}
\def\As                 {A^{str}}

\def\iso     {\xrightarrow{\sim}}

\def\unit       {\mathbf{1}}
\def\varphis    {\varphi^{str}}
\newcommand\G    {\texorpdfstring{$G$}{G}}  
\begin{document}

\thispagestyle{empty}
\begin{flushright}
    {\sf ZMP-HH/11-14}\\
    {\sf Hamburger$\;$Beitr\"age$\;$zur$\;$Mathematik$\;$Nr.$\;$415}\\[2mm]
    September 2011
\end{flushright}
\vskip 2.0em
\begin{center}\Large
Strictification of weakly equivariant Hopf algebras
\end{center}\vskip 1.4em
\begin{center}
Jennifer Maier, Thomas Nikolaus and Christoph Schweigert
\end{center}

\vskip 3mm

\begin{center}\it
   Fachbereich Mathematik, \ Universit\"at Hamburg\\
   Bereich Algebra und Zahlentheorie\\
   Bundesstra\ss e 55, \ D\,--\,20\,146\, Hamburg  
\end{center}
\vskip 2.5em

\begin{abstract}
A weakly equivariant Hopf algebra is a Hopf algebra $A$ with an action of a 
finite group $G$ up to inner automorphisms of $A$. We show that each weakly equivariant Hopf algebra can be replaced by a Morita equivalent algebra $\As$ with a strict action of $G$ and with a coalgebra structure
that leads to a  tensor equivalent representation category. However, the coproduct of this strictification cannot, in general, 
be chosen to be unital, so that a strictification
of the $G$-action can only be found on
a \emph{weak} Hopf algebra $\As$.
\end{abstract}

\section{Introduction}

This paper is a supplement to our paper \cite{mns2011}. In that paper we constructed a 3-dimensional equivariant topological field theory which is a generalization of the well-known Dijkgraaf-Witten theory \cite{DW90, FQ93}. Our generalization is equivariant with respect to a finite group $G$ (which was called $J$ there). Our motivation comes from  orbifold constructions in conformal field theory. 

It is well known that one can extract a modular category $\calc$ from a 3-dimensional topological field theory, at least up to some technical subtleties \cite[Chapter 4 \& 5]{BKLec},
involving properties of the dualities. A modular category is, 
in particular, a tensor category.  If the initial topological field theory is 
moreover $G$-equivariant, the category $\calc$ carries additionally a $G$-grading and an action of $G$ that is compatible with the tensor product. Such a structure is called a $G$-equivariant tensor category  resp. $G$-modular category \cite{kirI17, turaev2010}.

In general the action of the group $G$ on a $G$-modular category $\calc$ is given by tensor functors $\phi_g: \calc \to \calc$ together with 
compositors $\phi_g \circ \phi_h \iso \phi_{gh}$, subject to coherence laws 
 for threefold products. 
It has been demonstrated by M\"uger 
\cite[Appendix 5]{turaev2010} that one can replace $\calc$ by an equivalent category $\calc^{str}$ with a {\em strict} action of $G$, i.e. there the compositors are given by the identity: $\phi_g \circ \phi_h = \phi_{gh}$. 

Now consider the $G$-modular category $\calc$ which belongs to our equivariant Dijkgraaf-Witten theory mentioned at the beginning.
Although the category $\calc$ can relatively easily be described abstractly, it is very hard to work with it explicitly when it comes to orbifolding and showing modularity. Therefore in \cite[section 4]{mns2011} we realized $\calc$ as the representation category of a certain algebra $A$, which we called the equivariant Drinfel'd double. The fact that $\calc$ is a tensor category is reflected by the fact that $A$ is a Hopf algebra. Furthermore  there is also an algebraic structure on $A$ belonging to the $G$-action on the representation category. This structure is not just a $G$-action on $A$, as one might naively expect, but a \emph{weak} $G$-action, which is an action by Hopf algebra automorphisms $\varphi_g: A \to A$ such that $\varphi_g \circ \varphi_h$ equals $\varphi_{gh}$ only up to an inner automorphism of $A$. This weakening of the $G$-action reflects the fact that the action on the category is only weak in the sense that we 
have isomorphisms $\phi_g \circ \phi_h \stackrel\sim\to \phi_{gh}$ 
of functors rather than equalities. In order to accommodate the example of the algebra $A$, we had to introduce the notion of Hopf algebra with weak $G$-action (\cite[definition 4.13]{mns2011}), generalizing the notion of Hopf algebra with strict $G$-action considered before \cite{turaev2010,Vir02}. \\

In the light of M\"uger's observation that one can replace a $G$-equivariant tensor category $\calc$ by an equivalent category $\calc^{str}$ with strict 
$G$-action it is a natural question to ask whether one can replace a Hopf algebra $A$ with weak $G$-action by a Hopf algebra $A^{str}$ with strict 
$G$-action such that the representation categories are equivalent as tensor categories. A first result of this paper asserts that this is not possible in 
general, see Theorem \ref{prop:nostrict}. The reason is that the Hopf algebra
axioms are too rigid:
the tensor product of the representation category
is, in the case of Hopf algebras, directly inherited from
the underlying tensor product of vector spaces. 
Weak Hopf algebras \cite{BNSz99,BSz00, NV02} have been introduced
to provide a more flexible notion for the tensor product. 
Note that the qualifier weak here refers to a weakening of the bialgebra axioms (i.e.\ a weakening of the unitality of the coproduct or, 
equivalently, of the counitality of the product) and should 
not be confused with `weak $G$-action'. We refer to the appendix for
a table summarizing the situation.

Thus, a refined version of the question posed above would be whether one can replace a Hopf algebra $A$ with weak $G$-action by a 
{\em weak} Hopf algebra $A^{str}$ with {\em strict} $G$-action such that the representation categories are equivalent. The second main result of the present paper is to show that this is indeed possible, see Theorem \ref{maintheorem}. The given concrete construction of $A^{str}$ is inspired 
by M\"uger's strictification procedure 
\cite[Appendix 5]{turaev2010}
on the level of categories. Nevertheless we present it in an independent 
and elementary manner which requires no knowledge about orbifold categories 
and other constructions that enter in the categorical strictification.

\paragraph{Acknowledgements.} We thank Alexander Barvels 
and J\"urgen Fuchs for helpful discussions. 
TN and CS are partially supported by the Collaborative Research Centre 676 
``Particles, Strings and the Early Universe - the Structure of Matter 
and Space-Time'' and the cluster of excellence ``Connecting particles with the cosmos''. JM and CS are partially supported by the Research priority program 
SPP 1388 ``Representation theory''.
We are grateful to the anonymous referee for many useful suggestions.

\section{Equivariant Hopf algebras and their representation categories}

In the following, let $G$ be a finite group.

\begin{Definition}
Let $A$ be an (associative, unital) algebra over a
field $\KK$. A \emph{weak $G$-action} on $A$ 
consists of (unital) algebra automorphisms $\varphi_g \in \Aut(A)$,
one  for every element $g\in G$, and invertible 
elements $c_{g,h} \in A^\times$, one for every pair of elements 
$g,h \in G$, such that for all $g,h,k \in G$
the following conditions are satisfied:
\begin{align}\label{compositors}
  \varphi_g \circ \varphi_h = \Inn_{c_{g,h}} \circ \varphi_{gh}\qquad
  \varphi_g(c_{h,k}) \cdot c_{g,hk}  = c_{g,h} \cdot 
  c_{gh,k} \quad \text{ and } \quad c_{1,1} = 1 \,\,\, . 
\end{align}
Here $\Inn_x$ with $x$ an invertible element of $A$ denotes 
the algebra automorphism $a \mapsto xax^{-1}$.
A weak action of a group $G$ is called
\emph{strict} if  $c_{g,h} = 1$ for all pairs $g,h\in G$. 
\end{Definition}

\begin{Remark}
Note that our notion of a weak action $(\varphi_g, c_{g,h})$ of a group $G$ on an algebra $A$ corresponds to a weak action in the sense of \cite{BCM86} together with the normal cocycle
\begin{align*}
\sigma: \KK[G] \times \KK[G]& \to A^{\times}\\
(g,h)&\mapsto c_{g,h}
\end{align*}
 that fulfills the cocycle and the twisted module condition of \cite{BCM86}.
\end{Remark}

We first demonstrate how a weak $G$-action on an algebra
$A$ induces a categorical action (see \cite{mns2011} for the definition)
on the representation category $A\mod$. Here by $A\mod$ we denote the category of right modules over $A$; using left modules would lead to slightly more 
complicated formulas in the rest of the paper. We define for each 
element $g\in G$ a functor $\phi_g$
on objects by
$$^g(M,\rho) := (M,\rho \circ (\id_M \otimes \varphi_{g^{-1}}))$$
and on morphisms by the identity, $^{\!g}f = f$, and take, for the functorial
isomorphisms,
$\alpha_{g,h}(M,\rho):= \rho(\id_M \otimes c_{h^{-1},g^{-1}})$. One can check, that the cocycle condition in (\ref{compositors}) implies the equality
\begin{equation*}
\alpha_{gh,k} \circ \alpha_{g,h} = \alpha_{g,hk} \circ  {}^g\alpha_{h,k}.
\end{equation*}
We summarize this in the following lemma.

\begin{Lemma}\label{action}
Given a weak action of $G$ on a $\KK$-algebra $A$,
the functors $\phi_g$ and the natural transformations
$\alpha_{g,h}$  define a categorical action on 
the abelian category $A\mod$ of right $A$-modules.
\end{Lemma}
 
In the following we will mostly be interested in Hopf algebras. We therefore adapt the definition of a $G$-action to Hopf algebras.

\begin{Definition}\label{def:Hopf-action}
A weak $G$-action on a Hopf algebra $A$ is a weak $G$-action $((\varphi_g)_{g\in G},(c_{g,h})_{g,h\in G})$ on the underlying algebra which in addition satisfies the following properties:
\begin{itemize}
\item
$G$ acts by automorphisms of Hopf algebras.
\item The elements $(c_{g,h})_{g,h\in G}$ are group-like, i.e $\Delta(c_{g,h}) = c_{g,h}\otimes c_{g,h}$.
\end{itemize}
\end{Definition}

\begin{Remark}\label{rem:weak-Hopf-action}
Analogously, one can give the definition of a weak $G$-action on a weak 
Hopf algebra.
In that case, we require the elements
$c_{g,h}$ to be right grouplike in the sense of \cite{Vecs}.
By \cite[Corollary 5.2]{Vecs}, this amounts to all
$c_{g,h}$ being invertible and obeying
$\Delta(c_{g,h})=(c_{g,h}\otimes c_{g,h})\Delta(1)$.
 \end{Remark}

\begin{Lemma}
Given a weak action of $G$ on a Hopf algebra $A$,
the induced action on the tensor category $A\mod$ 
of right $A$-modules is by strict tensor functors and tensor transformations.
\end{Lemma}

We next turn to an algebraic structure that yields tensor categories with 
$G$-action and compatible $G$-grading, called 
$G$-equivariant tensor categories \cite{kirI17}.

\begin{Definition}\label{GHopf}
A \emph{$G$-Hopf algebra} over $\KK$ is a Hopf algebra 
$A$ with a weak $G$-action $((\varphi_g)_{g\in G},(c_{g,h})_{g,h\in G})$ 
and a $G$-grading $A = \bigoplus_{g\in G} A_{g}$ such that:
\begin{itemize}
\item 
The algebra structure of $A$ restricts to
the structure of an associative algebra on
each homogeneous component so that $A$ is the direct 
sum of the components $A_g$ as an algebra.

\item The action of $G$ is compatible with the grading,
i.e.\ $\varphi_g(A_h) \subset A_{ghg^{-1}}$.

\item 
The coproduct $\Delta:\ A\to A\otimes A$ respects
the grading, i.e.\
$$\ \Delta (A_g) \subset \bigoplus_{p,q\in G,pq=g} A_p\otimes
A_q \,\, . $$

\end{itemize}
\end{Definition}

\begin{Remark}\label{normalized}
\begin{enumerate}
\item
For the counit $\epsilon$ and the antipode $S$ of
a $G$-Hopf algebra, the compatibility relations
with the grading $\epsilon(A_g) = 0$ for $g \ne 1$ and 
$S(A_g) \subset A_{g^{-1}}$ are immediate consequences of
the definitions.

\item The restrictions of the structure maps
endow the homogeneous component $A_1$ of $A$  
with the structure of a Hopf algebra with a weak
$G$-action.

\item $G$-Hopf algebras with \emph{strict} $G$-action 
have been considered under the name ``$G$-crossed Hopf coalgebra''
in \cite[Chapter VII.1.2]{turaev2010}.

\item Hopf algebras with weak $G$-action give a special case of $G$-Hopf 
algebra, where the grading is concentrated in degree $1$. Thus all results of this paper imply analogous results where the term $G$-Hopf algebra is 
replaced by Hopf algebra with weak $G$-action.
\end{enumerate}
\end{Remark}

The category $A\mod$ of finite-dimensional modules
over a $G$-Hopf algebra inherits a natural (left and right) duality
from the duality of the underlying category of 
$\KK$-vector spaces. The weak action described in
Lemma \ref{action} is even a monoidal action,
since $G$ is required to act by Hopf algebra morphisms.
A grading on $A\mod$ can be given by taking 
$(A\mod)_g = A_g\mod$ as the $g$-homogeneous component. 
{}From the properties of a $G$-Hopf algebra one can finally deduce 
that the tensor product, duality and grading are compatible with 
the $G$-action. We have thus arrived at the following statement:

\begin{Lemma}{\textnormal{\cite[Lemma 4.15]{mns2011}}}\label{G-equivariant}
The category of representations of a $G$-Hopf algebra inherits the natural structure of a $\KK$-linear, abelian 
$G$-equivariant tensor category with dualities.
\end{Lemma}

A similar result holds for  $G$-weak Hopf algebras.

\section{Strictification of the group action}

The action of the group $G$ on a $G$-equivariant tensor category $\calc$ can always be strictified (see \cite[Appendix 5]{turaev2010}), i.e.\ there is an equivalent $G$-equivariant tensor category $\calc^{str}$ with strict $G$-action (all compositors are identities). 
If one starts  with the representation category of a $G$-Hopf algebra $A$, it 
is natural to ask whether this strictification leads to the representation 
category of another $G$-Hopf algebra with strict $G$-action. We will make 
this precise in the next definition. A $G$-equivariant functor between $G$-equivariant tensor categories is a tensor functor 
$F$ together with natural isomorphisms
\[\psi_g: F( {}^gM) ~\iso~ {}^gF(M)\]
such that for every pair $g,h \in G$ the obvious coherence diagrams of morphisms from $F(^{gh}M)$ to\; $^{gh}F(M)$ commute. See also \cite[Appendix 5, Def. 2.5]{turaev2010}.

\begin{Definition} \label{def:strictification}
\begin{enumerate}
\item
Let $A$ be a Hopf algebra with weak $G$-action. A \emph{strictification} of $A$ is a weak Hopf algebra $B$ with \emph{strict} $G$-action and 
an equivalence 
\[A\mod \iso B\mod\]
of tensor categories with $G$-action.
\item
Let $A$ be a $G$-Hopf algebra. A \emph{strictification} of $A$ is a $G$-weak Hopf 
algebra $B$ with \emph{strict} $G$-action and an equivalence 
\[A\mod \iso B\mod\]
of $G$-equivariant tensor categories.
\end{enumerate}

\end{Definition}

We will now show that it is in general not possible to find a strictification that is a Hopf algebra, rather than a weak Hopf algebra. This shows that we really have to allow for weak Hopf algebras as strictifications. In the next chapter we then show that a strictification as a weak Hopf algebra always exists. \\

Consider the weak action of $\ZZ/2\times\ZZ/2 = \{1,t_1,t_2,t_1t_2\}$ on the group algebra $ \mathbb{C}[\ZZ/2]$ of $\ZZ/2 = \{1,t\}$ given by

\begin{equation*}
\varphi_g =  \id \quad\quad \textnormal{for all } g \in \ZZ/2\times\ZZ/2
\end{equation*}
and non-trivial compositors given by the grouplike elements $c_{g,h}\in \mathbb{C}[\ZZ/2]$ as in the following table:

\begin{equation}\label{table}
\begin{tabular}{l|c|c|c|c}
$h \backslash g$ &1& $t_1$ & $t_2$ & $t_1t_2$\\
\hline
1&1&1&1&1
\\
\hline
$t_1$&1&$t$&$t$&1
\\
\hline
$t_2$&1&1&1&1
\\
\hline
$t_1t_2$&1&$t$&$t$&1
\end{tabular}
\end{equation}

In \cite[Section 3.1]{mns2011} we showed how weak actions correspond to extensions of groups together with the choice of a set theoretic section. In this case, the relevant extension is given
by the exact sequence of groups
\begin{equation*}
\ZZ/2 \to D_4 \to \ZZ/2\times \ZZ/2\, ,
\end{equation*}
where $D_4$ denotes the dihedral group of order 8. The inclusion of $\ZZ/2$ into $D_4$ is given by mapping the nontrivial $t$ element of $\ZZ/2$ to the rotation by $\pi$. The projection to $\ZZ/2\times \ZZ/2$ is given by mapping the rotation $a \in D_4$ by $\frac{\pi}{2}$ to the first generator $t_1$ and the reflection $b \in D_4$ to the second generator $t_2$. The set theoretic section is defined by $s:\ZZ/2\times \ZZ/2 \to D_4$ with $s(1) = 1, s(t_1) = a, s(t_2) = b, s(t_1t_2) = ab$.
\begin{Theorem}\label{prop:nostrict}
There is no strictification  as a Hopf algebra of $\mathbb{C}[\ZZ/2]$ with the weak $\ZZ/2\times\ZZ/2$-action with compositors as displayed in \eqref{table}.
\end{Theorem}

\begin{Remark}
Note that the algebra $\mathbb{C}[\ZZ/2]$ is not a priori endowed with a grading by $\ZZ/2$. We can consider it as being trivially graded.  
\end{Remark}

For the proof of proposition \ref{prop:nostrict} we need 
the following elementary facts:

\begin{Lemma}\label{hilfslemma}
Let $A = \mathbb{C}[G]$ be the complex group algebra of a finite abelian group $G$. 
\begin{enumerate}
\item \label{punkteins}
Let $A'$ be an arbitrary Hopf algebra. If $A\mod \cong A'\mod$ as tensor categories, then $A \cong A'$ as algebras (not necessarily as Hopf algebras).
\item\label{punktzwei}
The natural endomorphisms of the identity functor $\textnormal{Id}: A\mod \to A\mod$ are given by the action of elements in $A$. More precisely there is an 
isomorphism of algebras
\begin{equation*}
A \iso \textnormal{End}(\textnormal{Id})_{A\mod} \,\,\, .
\end{equation*}
\item\label{punktdrei}
Let $\varphi: A \to A$ be an algebra automorphism such that the restriction functor $res_\varphi: A\mod \to A\mod$ is naturally isomorphic to the identity functor. Then $\varphi = \id$.
\end{enumerate}
\end{Lemma}
\begin{proof}
1.) By the reconstruction theorem we know that we can recover the Hopf algebra $A$ as endomorphisms of the fibre functor $F: A\mod \to \mathbb{C}\mod$ and $A'$ as endomorphisms of the fibre functor $G: A\mod \iso A'\mod \to \mathbb{C}\mod$. Now we claim that the underlying functors of $F$ and $G$ are naturally isomorphic. To this end note that for each simple representation $V_i$ of $A$ we have $V_i^n \cong 1$ where $n$ is the order of the group. Thus we have $F(V_i) \cong \mathbb{C} \cong G(V_i)$ by the fact that $F$ and $G$ are tensor-functors. But it is easy to see that the $\mathbb{C}$-linearity and the fact that $A\mod$ is semisimple then already show that $F$ and $G$  are isomorphic as functors
between abelian categories. This implies that $A \cong \End(F) \cong \End(G) \cong A'$. Note that the functors $F$ and $G$ still might have different tensor functor structures, leading to different Hopf algebra structures on $A$ and $A'$. \\

\noindent 2.) This follows from the fact that $A$ is abelian and from the 
fact that the center of an algebra is isomorphic to the endomorphisms of the 
identity functor on its representation category. \\

\noindent 3) The functor $res_\varphi$ is an equivalence of categories. Hence it sends simple objects to simple objects. That means it acts on simple characters $\chi: G \to \mathbb{C}^*$. By the fact that this functor is naturally isomorphic to the identity this action has to be trivial. Hence we know $\chi \circ \varphi  = \chi$ for each character $\chi$.
Because $G$ is abelian, the characters form a basis of 
the dual space $A^*$. Thus $\varphi^* = \id$ which implies $\varphi = \id$.

\end{proof}

\begin{proof}[Proof of theorem \ref{prop:nostrict}]
Assume that there is a Hopf algebra $H$ with a strict action of $\ZZ/2$ by 
Hopf algebra automorphisms $\varphi_{g}$ together with an equivalence
of categories $A\mod\to H\mod$. 
By lemma \ref{hilfslemma}\eqref{punkteins} we know that the underlying algebra of $H$ is isomorphic to $\mathbb{C}[\ZZ/2]$. We choose an isomorphism and transport the action $\varphi_{g}$ on $H$ to an action $\varphi'_{g}$ on $\mathbb{C}[\ZZ/2]$ (which is now only an action by algebra automorphisms and not necessarily by 
Hopf algebra automorphisms). By assumption there are now natural isomorphisms $res_{\varphi'_g} \iso res_{\varphi_g} = \text{Id}$ hence by lemma \ref{hilfslemma}\eqref{punktdrei} we have $\varphi'_g = \id$. 

Now we have both times the trivial action on the Hopf algebra $\mathbb{C}[\ZZ/2]$, 
once with the nontrivial compositors $c_{g,h}$ as displayed in table \eqref{table} above and once with the trivial compositors. By Lemma \ref{hilfslemma}\eqref{punktzwei}, an isomorphism between the two induced actions on the representation categories is induced by invertible elements $(a_g \in \mathbb{C}[\ZZ/2])_{g \in \ZZ/2\times \ZZ/2}$ such that 
\begin{align}\label{elements}
a_{gh} \cdot c_{g,h} = a_g \cdot a_h & &\text{for all}\; g,h \in \ZZ/2\times \ZZ/2
\end{align} 
We show that such elements can not exist:
Assume, there are invertible elements $(a_g \in \mathbb{C}[\ZZ/2])_{g \in \ZZ/2\times \ZZ/2}$ that fulfill \eqref{elements}. In particular we have, by setting $g= h = 1$ in \eqref{elements},
$a_1^2 = a_1$,
and since the elements $a_g$ are invertible, it follows that $a_1 = 1$. One concludes similarly, by setting $g = h = t_2$ resp. $g= h= t_1t_2$, that $a_{t_2}^2 = 1$ and $a_{t_1t_2}^2 = 1$. Now if we set $g= t_1$ and $h = t_2$ in \eqref{elements} and take the square of the resulting equation, we get
$a_{t_1}^2 = 1$,
but clearly those elements don't fulfill the equation $a_{1}t = a_{t_1}^2 $, which is \eqref{elements} with $g=h=t_1$.
This contradicts the existence of the 
strictification $H$ of a Hopf algebra.
\end{proof}

\section{Existence of a strictification}

In this section, we will successively prove the following theorem which
holds for Hopf algebras over an arbitrary field $\KK$.
\begin{Theorem}\label{maintheorem}
\begin{enumerate}
\item
For any Hopf algebra with weak $G$-action there exists a strictification in the sense of definition \ref{def:strictification}(1).
\item
For any $G$-Hopf algebra there exists a strictification in the sense of definition \ref{def:strictification}(2).
\end{enumerate}
\end{Theorem}

Note that the first part of Theorem \ref{maintheorem} follows from the second part if we consider a Hopf algebra with weak $G$-action as a $G$-Hopf algebra with grading concentrated in degree $1$, see also Remark 
\ref{normalized}(4). Therefore we will only prove the 
second part.

In the following let $A$ be a $G$-Hopf algebra with unit $1_A$, counit $\epsilon_A$, coproduct $\Delta_A$ and a weak $G$-action $((\varphi_g)_{g\in G},(c_{g,h})_{g,h\in G})$. The plan of this section is to construct step by step a 
strictification $\As$. 

In section \ref{algebra} we construct $\As$ as an algebra, in section \ref{weakHopf} we endow it with a weak Hopf algebra structure and finally in section \ref{equivariance} we turn it into a $G$-weak Hopf algebra with strict $G$-action. Along the way, we also provide the necessary equivalences of the representation categories 
\begin{equation*}
F:A\mod \overset{\sim}{\rightarrow}   \As\mod
\end{equation*}
and show that they preserve all the structure involved. This implies that 
$A^{str}$ is a strictification, which proves Theorem \ref{maintheorem}.

\subsection{The algebra}\label{algebra}

In the following we use the notation $\KK(G)$ for the $\KK$-vector space of functions on the finite group $G$, with distinguished basis $(\delta_{g})_{g\in G}$. By $\KK[G]$ we denote the $\KK$-vector space underlying the group algebra with basis $(g)_{g\in G}$.
\begin{Definition}
Set $\As= \KK(G)\otimes_\KK A \otimes_\KK \KK[G]$ as a vector space and define a multiplication on the generators of $\As$ by
\begin{align}\label{str_product}
(\delta_g \otimes a \otimes h)(\delta_{g'} \otimes a' \otimes h') = \delta(gh',g')(\delta_{g'} \otimes a \varphi_{h}(a')c_{h,h'} \otimes hh')
\end{align}
where $\delta(gh',g')$ is the Kronecker delta, i.e. $\delta(gh',g')=1$ if $gh'=g'$ and $\delta(gh',g')=0$ otherwise. This multiplication has the unit
\begin{align}\label{str_unit}
1 = \sum_{g\in G} \delta_g \otimes 1_A \otimes 1.
\end{align}
\end{Definition}
It can easily be checked that the product and the unit defined in (\ref{str_product}) and (\ref{str_unit}) endow $\As$ with the structure of an associative unital algebra.\\

We next define a functor $F:A\mod\to\As\mod$:
Let $M$ be an object in $A\mod$. Define an object in $\As\mod$ which is $M\otimes_\KK \KK[G]$ as a vector space and has the following right action of the algebra $\As$:
On an element of the form $(m\otimes k)$ with $m\in M, k \in G$, the action of $(\delta_g\otimes a \otimes h)$ reads:
\begin{align}\label{action_of_Astr}
(m\otimes k).(\delta_g\otimes a \otimes h) := \delta(kh,g) (m.\varphi_k(a)c_{k,h}\otimes g)
\end{align}
One checks that this really defines a right action of $\As$. For a morphism $f \in \Hom_{A}(M,N)$ we consider the morphism  $f \otimes \id_{\KK[G]}\in \Hom_{A^{str}}(M\otimes \KK[G] ,N\otimes \KK[G])$. Together this defines a functor:
\begin{align}\label{equivalence}
F:A\mod &\to \As\mod 
\end{align}

\begin{Proposition}\label{prop:equivalence}
The functor $F$ is an equivalence of abelian categories.
\end{Proposition}

\begin{proof}
We show that $F$ is essentially surjective and fully faithful.

For the essential surjectivity, note that an object $N$ in $\As\mod$ has a $G$-grading $N = \bigoplus_{g\in G}N_g$ with $N_g := N.(\delta_g \otimes 1_A \otimes 1)$. Endow the subspace $N_1 := N.(\delta_1\otimes 1_A \otimes 1) \subset N$ with the structure of an $A$-module by setting $n.a := n.(\delta_1\otimes a \otimes 1)$ for an element $n \in N_1$ and $a \in A$.

Define the $\KK$-linear map
\[\Theta: F(N_1) = N_1\otimes \KK[G] \to N\]
by
\[(n\otimes g) \mapsto n.(\delta_g \otimes 1_A \otimes g).\]
It is easy to see that $\Theta$ is an isomorphism with inverse $n \mapsto \sum_{g\in G} n.(\delta_1 \otimes (c_{g^{-1},g})^{-1}\otimes g^{-1}) \otimes g$.
To see that $\Theta$ is a morphism in $\As\mod$, note that the action of $\As$ on $F(N_1)$ is given by
\[(n\otimes k).(\delta_g \otimes a \otimes h) = \delta(kh,g) n.(\delta_1 \otimes \varphi_k(a)c_{k,h}\otimes 1)\otimes g\]
Hence we have
\begin{align*}
\Theta((n\otimes k).(\delta_g \otimes a \otimes h)) &= \delta(kh,g)\Theta(n.(\delta_1 \otimes \varphi_k(a)c_{k,h}\otimes 1)\otimes g)\\
&= \delta(kh,g)n.(\delta_1 \otimes \varphi_k(a)c_{k,h}\otimes 1)(\delta_g \otimes 1_A \otimes g)\\
&= \delta(kh,g)n.(\delta_g \otimes \varphi_k(a)c_{k,h}\otimes g)\\
\end{align*}
and
\begin{align*}
\Theta(n\otimes k).(\delta_g \otimes a \otimes h)
&= n.(\delta_k \otimes 1_A \otimes k)(\delta_g\otimes a\otimes h)\\
&= \delta(kh,g)n.(\delta_g\otimes\varphi_k(a)c_{k,h}\otimes g).
\end{align*}
This shows that $F(N_1) \cong N$ as $A^{str}$-modules and thus essential surjectivity.

It is clear that $F$ is faithful. In order to see that $F$ is also full, consider for two $A$-modules $M,N$ a morphism $f \in \Hom_{A^{str}}(F(M),F(N))$. We have
\begin{align*}
f(m\otimes k).(\delta_g \otimes 1_A \otimes 1) = f((m\otimes k).(\delta_g \otimes 1_A \otimes 1)) = \delta(k,g) f(m\otimes k),
\end{align*}
so $f(m\otimes k)\in N \otimes \KK k$
and we have $f = \sum_{g\in G} f_{g}\otimes \id_{\KK g}$ for some $f_g \in \Hom_{\KK}(M,N)$. From the definition \eqref{action_of_Astr} of the action of $\As$, it is clear that $f_1 \in \Hom_{A}(M,N)$. Now, since $f$ commutes with the action of $\As$, we have
\begin{align*}
\delta(kh,g)f_{g}(m.\varphi_k(a)c_{k,h})\otimes g & = f((m\otimes k).(\delta_g\otimes a \otimes h))\\
& = f(m\otimes k).(\delta_g \otimes a \otimes h)\\
& = \delta(kh,g)f_k(m).\varphi_k(a)c_{k,h}\otimes g
\end{align*}
and (by setting $a = 1_A$, $k=1$ and $h=g$) we get $f_g = f_1$ for all $g\in G$ and therefore $f$ is of the form $f =  f_1 \otimes \id_{\KK[G]}$.
\end{proof}

\subsection{The weak Hopf algebra structure}\label{weakHopf}

We need the strictification algebra $\As$ to have more structure in order for its representation category to be a tensor category. In fact, we want it to be a weak bialgebra. For the definition and properties of weak bialgebras, see e.g. \cite{nikshych2003invariants}. 
\begin{Proposition}
The linear maps $\Delta:\As \to \As \otimes \As$ and $\epsilon:\As \to \KK$ 
defined on the generators of $\As$ by
\begin{align*} 
\Delta(\delta_g \otimes a \otimes h) &= \sum_{(a)} (\delta_g \otimes a_{(1)} \otimes h)\otimes(\delta_g \otimes a_{(2)} \otimes h)\,\,\, , \\
\epsilon(\delta_g \otimes a \otimes h) &= \epsilon_A(a)
\end{align*}
endow $\As$ with the structure of a weak bialgebra.
Furthermore, the linear map $S:\As\to \As$ given by

\begin{align*}
S(\delta_g \otimes a \otimes h) = (\delta_{gh^{-1}} \otimes c_{h^{-1},h}^{-1}\cdot \varphi_{h^{-1}}\big(S_A(a)\big)\otimes h^{-1})
\end{align*}
is an antipode for $\As$, where $S_A$ is the antipode of A. 
\end{Proposition}

\begin{proof}
The maps $\Delta$ and $\epsilon$ are a coassociative coproduct and a counit 
on $\As$, as they are just the structural maps of the tensor product 
coalgebra of $\KK(G)$, $A$  and $\KK[G|$ (where we consider the diagonal 
coproduct on both $\KK(G)$ and $\KK[G]$). We show that $\Delta$ is also a 
morphism of algebras, i.e.\ that
\begin{align}\label{alg.mor}
(m\otimes m)\circ(\id\otimes\tau\otimes\id)(\Delta\otimes \Delta) = \Delta \circ m. 
\end{align}
If we plug in two elements $(\delta_g \otimes a \otimes h), (\delta_{g'} \otimes a' \otimes h')$, we get for the left hand side of (\ref{alg.mor})
\begin{align*}
&\quad\sum_{(a)} (\delta_g \otimes a_{(1)} \otimes h)\cdot(\delta_{g'} \otimes a_{(1)}' \otimes h') \otimes (\delta_g \otimes a_{(2)} \otimes h)\cdot(\delta_{g'} \otimes a_{(2)}' \otimes h')\\
& = \sum_{(a)} \delta(gh',g')(\delta_{g'} \otimes a_{(1)} \varphi_h(a_{(1)}')c_{h,h'} \otimes hh')\otimes (\delta_{g'} \otimes a_{(2)} \varphi_h(a_{(2)}')c_{h,h'} \otimes hh') \,\,\, , 
\end{align*}
and for the right hand side
\begin{align*}
&\quad\delta(gh',g')\Delta(\delta_{g'} \otimes a \varphi_h(a')c_{h,h'} \otimes hh')\\
& = \delta(gh',g')\sum_{(a)}(\delta_{g'} \otimes (a \varphi_h(a')c_{h,h'})_{(1)} \otimes hh') \otimes (\delta_{g'} \otimes (a\varphi_h(a')c_{hh'})_{(2)} \otimes hh')\\
& = \sum_{(a)}\delta(gh',g')(\delta_{g'} \otimes a_{(1)} \varphi_h(a_{(1)}')c_{h,h'} \otimes hh')\otimes (\delta_{g'} \otimes a_{(2)} \varphi_h(a_{(2)}')c_{h,h'} \otimes hh') \,\,\, . 
\end{align*}
All equations follow just by definition, except for the last one, where we used that the coproduct in $A$ is a morphism of algebras, that the elements $c_{g,h}$ are group-like and that the action of $G$ on $A$ is a coalgebra-morphism.\\
Further equations concerning the compatibilities of the product with the 
counit, the coproduct with the unit and the antipode can be checked directly.
\end{proof}

In a weak Hopf algebra H the target and source counital maps are defined on an element $h \in H$ by
\begin{align*}
\epsilon_t(h) & := (\epsilon\otimes \id_H)(\Delta(1)(h  \otimes 1))\\
\epsilon_s(h) & := (\id_H \otimes\epsilon)((1  \otimes h)\Delta(1))
\end{align*}
The maps $\epsilon_t$ and $\epsilon_s$ are idempotents. The image of $H$ under them are called the target and source counital subalgebras 
\begin{align*}
H_t &:= \epsilon_t(H)\\ 
H_s &:= \epsilon_s(H).
\end{align*} 
The category of right modules over $H$ can be endowed with the structure of
a tensor category, where the tensor product of two modules $M,N$ is defined via the coproduct on the following vector space:
\[M\bar\otimes N := (M\otimes_{\KK} N) \Delta(1) \,\,\, . \]
The tensor unit is the source counital subalgebra $H_s$
with $H$-action given by $z.h := \epsilon_s(zh)$ for $h\in H, z \in H_s$.

\begin{Lemma}\label{counital}
For the algebra $\As$, the target and source counital maps are given by
\begin{align}
\epsilon_t(\delta_g \otimes a \otimes h) & = \epsilon_A(a)(\delta_{gh^{-1}}\otimes 1_A \otimes 1)\\
 \epsilon_s(\delta_g \otimes a \otimes h) & = \epsilon_A(a)(\delta_g\otimes 1_A \otimes 1)
\end{align}
and the target and source counital subalgebras are
\[A^{str}_t \cong A^{str}_s \cong \KK(G)\,\,\, . \]
\end{Lemma}

\begin{proof}
We calculate $\epsilon_t$ on an element $(\delta_g\otimes a \otimes h) \in \As$:
\begin{align*}
\epsilon_t(\delta_g\otimes a \otimes h)& = (\epsilon\otimes \id)(\Delta(1)((\delta_g\otimes a \otimes h)  \otimes 1))\\
& = (\epsilon \otimes \id) ((\delta_g\otimes a \otimes h) \otimes (\delta_{gh^{-1}}\otimes 1_A \otimes 1) )\\
& = \epsilon_A(a)(\delta_{gh^{-1}} \otimes 1_A \otimes 1) \,\,\, . 
\end{align*}
The calculation for the source counital map is completely parallel.
Choose a basis $(a_i)_{i\in I}$ of the algebra $A$ with $a_j = 1_A$ for a fixed $j\in I$, then a general element $b \in \As$ is of the form
\[b = \sum_{g,h\in G, i\in I} \lambda(g,h,i)(\delta_g \otimes a_i \otimes h)\]
with $\lambda(g,h,i) \in \KK$. We have:
\begin{align*}
\Delta (b)& = \sum_{g,h\in G,i\in I} \sum_{(a_i)}\lambda(g,h,i) (\delta_g \otimes (a_i)_{(1)} \otimes h)\otimes(\delta_g \otimes (a_i)_{(2)} \otimes h)\\
\Delta(1)(b\otimes 1) &= \sum_{g,h\in G,i\in I} \lambda(g,h,i) (\delta_g \otimes a_i \otimes h)\otimes(\delta_g \otimes 1_A \otimes h)
\end{align*}
By equating coefficients, we get 
$\lambda(g,h,i) = 0$ for $h\ne 1, i \ne j$ and therefore:
\[\As_t = \langle\delta_g \otimes 1_A \otimes 1, g \in G\rangle 
\cong \KK(G) \,\,\, . \]
An analog calculation shows the same result for $\As_s$.
\end{proof}

\begin{Proposition}
The equivalence $F:A\mod\iso\As\mod$ can be promoted to an equivalence of tensor categories.
\end{Proposition}

\begin{proof}
The tensor unit in the representation category of the weak Hopf algebra $\As$ is given by
the source counital subalgebra, which is by lemma \ref{counital} isomorphic to $\KK(G)$. The weak Hopf algebra $\As$ acts on the source counital subalgebra as follows: for an element $(\delta_k \otimes 1_A \otimes 1) \in \As_s$ and $(\delta_g\otimes a \otimes h)\in \As$, we have
\[(\delta_k \otimes 1_A \otimes 1).(\delta_g\otimes a \otimes h) =
\epsilon_s((\delta_k \otimes 1_A \otimes 1)(\delta_g\otimes a \otimes h)) 
= \delta(kh,g)\epsilon_A(a)(\delta_{g} \otimes 1_A \otimes 1)
\,\,\, . \]
The action of an element $a \in A$ on the tensor unit $\KK$ in $A\mod$ is by multiplication with $\epsilon_A(a)$. So we get for the image of the tensor unit under $F$ the vector space $\KK\otimes \KK[G]$ with $\As$-action
\[(\lambda \otimes k).(\delta_g\otimes a \otimes h) = \delta(kh,g)\epsilon_A(a)\lambda\otimes g \,\,\, . \]
We clearly have an isomorphism $F(\mathbf{1}) \to \mathbf{1}$ in $\As\mod$ 
given by 
\[\eta_0:(\lambda \otimes k) \mapsto \lambda \delta_k.\]
Let $M,N \in A\mod$. We have 
\begin{align*}
F(M)\bar \otimes F(N) = \langle (m\otimes g \otimes n \otimes g), m\in M, n\in N, g \in G \rangle.
\end{align*}
Thus the linear map
\[\eta_2(M,N): (m\otimes g\otimes n \otimes g) \mapsto (m\otimes n \otimes g)\]
is an isomorphism $F(M)\bar \otimes F(N) \iso F(M \otimes N)$.
It can be seen to commute with the action of $\As$ and is natural in $M,N$.
Moreover the isomorphisms $\eta_2$ clearly satisfy the coherence axioms for 
three objects. We have therefore established that $(F,\eta_0,\eta_2)$ is a 
tensor functor.
\end{proof}

\subsection{\G-action and \G-grading}\label{equivariance}

We will now define a $G$-equivariant structure on $\As$ that induces a 
$G$-equivariant structure on the category $\As\mod$.
The last step of proving theorem \ref{maintheorem} is then to show that the 
categories $A\mod$ and $\As\mod$ are even equivalent as $G$-equivariant 
categories. 

\begin{Definition}
On the weak Hopf algebra $\As$ we have a strict left action $\varphis$ of 
the group $G$ given by translation in the first factor. Explicitly, an 
element $g'\in G$ acts on an element $(\delta_g\otimes a \otimes h) \in \As$ by
\[\varphis_{g'}(\delta_g \otimes a \otimes h) 
= (\delta_{g'g}\otimes a \otimes h) \,\,\, . \]
\end{Definition}

The strict $G$-action on $\As$ gives us a strict left $G$-action on the category $\As\mod$ by setting $\phi_g(M,\rho) = (M, \rho\circ(\id_M \otimes\varphis_{g^{-1}} )$. We will now establish, that the equivalence $A\mod \cong\As$ is compatible with the $G$-actions. 

\begin{Proposition}
The equivalence $F:A\mod \iso \As\mod$ given in (\ref{equivalence}) respects the $G$-action of the two categories, i.e. for every element $g \in G$ there are natural isomorphisms
\[\psi_g: F( {}^gM) ~\iso~ {}^gF(M)\]
such that for every pair $g,h \in G$ the obvious coherence diagrams of morphisms from $F(^{gh}M)$ to\; $^{gh}F(M)$ commute.
\end{Proposition}

\begin{proof}
For $M \in A\mod$ consider the linear map $\psi_g: M \otimes \KK[G] \to M \otimes \KK[G]$ defined by
\begin{align*}
\psi_g: (m\otimes k) \mapsto (m.c_{g^{-1},k}\otimes g^{-1}k). 
\end{align*}

We first show that $\psi_g$ is a morphism of $\As$-modules. To distinguish the actions on the different modules we use the notation ``$\star$'' for the $\As$-action on $F(^g M)$ and $^g F(M)$, ``$\odot$'' for the action on $F(M)$ and ``$.$'' for the $A$-action on $M$.
\begin{align*}
\psi_g\big( &(m \otimes k) \star (\delta_x \otimes a \otimes h) \big) \\
&= \psi_g\big( \delta(kh,x) (m.\varphi_{g^{-1}}(\varphi_{k}\big(a)c_{k,h}\big) \otimes x) \big) \\
&= \delta(kh,x)~ m.\varphi_{g^{-1}}(\varphi_{k}\big(a)c_{k,h}\big)c_{g^{-1}, kh} \otimes g^{-1}kh \\
&= \delta(kh,x) ~ m.c_{g^{-1},k}\varphi_{g^{-1}k}(a)(c_{g^{-1},k})^{-1} \varphi_{g^{-1}}(c_{k,h}) c_{g^{-1}, kh} \otimes g^{-1}kh \\
&= \delta(g^{-1}kh,g^{-1}x) ~ m.c_{g^{-1},k}\varphi_{g^{-1}k}(a)c_{g^{-1}k,h}  \otimes g^{-1}kh  \\
&= (m.c_{g^{-1},k} \otimes g^{-1}k) \odot(\delta_{g^{-1}x} \otimes a \otimes h) \\
&= \psi_g(m \otimes k) \star (\delta_x \otimes a \otimes h) \,\,\, .
\end{align*}
Moreover we have to verify that the $\psi_g$ satisfy a coherence condition for two indices $g$ and $h$. This condition can be checked similarly to the above 
computation, using the cocycle condition for the $c_{g,h}$.
\end{proof}

\begin{Definition}
We define a $G$-grading in the sense of \ref{GHopf} on the algebra $\As$ by:
\begin{align}
(\As)_h =  \bigoplus_{g \in G} \left(\KK(\delta_g)\otimes A_{g^{-1}hg}\right) 
\otimes \KK[G] \,\,\, . 
\end{align}
\end{Definition}

The following Lemma can then be checked directly.
\begin{Lemma}
The algebra $\As$ is a $G$-weak Hopf algebra with strict $G$-action, i.e a weak Hopf algebra with strict $G$-action and compatible $G$-grading.
\end{Lemma}

Note that the grading on $A$ resp. $\As$ gives a grading on the representation category by $(A\mod)_h := A_h\mod$ resp. $(\As\mod)_h = (\As)_h\mod$.
\begin{Proposition}
The equivalence $F:A\mod \iso \As\mod$ given in (\ref{equivalence}) respects the $G$-grading of the two categories, i.e. for every element $h \in G$, we have
\[F((A\mod)_h) \subset (\As\mod)_h.\]
\end{Proposition}

\begin{proof}
Let $M \in (A\mod)_h = A_h\mod$. We know that the action by the unit 
$(\unit_{A})_h$ of $A_h$ is the identity on $M$.
We need to show, that the $h$-component of the unit in $\As$, which is $\unit_{h} = \sum_{g\in G} (\delta_g \otimes (1_A)_{g^{-1}hg}\otimes 1)$, acts as an idempotent on $F(M)=M\otimes\KK[G]$. In fact it even acts as the identity: for any element of the form $(m\otimes k)$, we have:
\begin{align*}
(m\otimes k). \left(\sum_{g\in G} \delta_g \otimes (1_A)_{g^{-1}hg}\otimes 1\right)
&= \sum_{g\in G} (\delta(k,g)m.\varphi_k( (1_A)_{g^{-1}hg})c_{k,1}\otimes k)\\
&= (m.\varphi_k( (1_A)_{k^{-1}hk})\otimes k)\\
&= (m.(1_A)_{h}\otimes k)\\
& = (m\otimes k)
\end{align*}
where in the first equality we used the definition of the action of $\As$ 
on $M\otimes \KK[G]$ given in (\ref{action_of_Astr}), in the third equality 
the fact that $G$ acts by unital algebra morphisms and in the last equality
that $M$ is in the $h$-component of $A\mod$.\\
So we have $F(M).\unit_h = F(M)$; therefore $F(M)$ lies in the component 
$(\As\mod)_h$.
\end{proof}

\section{Equivariant R-Matrix and ribbon-element}

In \cite{mns2011} we considered $G$-equivariant categories with a $G$-braiding and a $G$-twist as additional data ($G$-ribbon categories). For the definition see \cite{turaev2010,kirI17}. Since those categories were our main motivation to study the strictification in terms of algebras, we want to say a few words about the $G$-ribbon structure.

The definition of a $G$-equivariant $R$-matrix is rather involved even in the 
strict Hopf algebra case. We will refrain here from stating the axioms for it 
explicitly, but we will instead make an equivalent definition:

\begin{Definition}\label{def:ribbon}
Let $A$ be a $G$-(weak) Hopf algebra. 
\begin{enumerate}
\item
A \emph{$G$-equivariant $R$-matrix} is an element $R = R_1 \otimes R_2 \in \Delta^{op}(1)(A\otimes A)\Delta(1)$ such that for $V\in (A\mod)_g$, $W\in A\mod$, the map
\begin{align*}
c_{VW}:~ &V\otimes W ~\to~ {}^{g}W \otimes V\\
& ~v\otimes w ~\mapsto~ w.R_2 \otimes v.R_1
\end{align*}
is a $G$-braiding, in particular a morphism of $A$-modules.
\item A \emph{$G$-twist} is an invertible element $\theta \in A$ such that for every object $V\in (A\mod)_g$ the induced map
\begin{align*}
\theta_V:~&V ~ \to ~{}^{g}V\\
& ~v ~\mapsto~ v.\theta^{-1}
\end{align*}

is a $G$-twist in $A\mod$.
\end{enumerate}
A $G$-\emph{(weak) ribbon-algebra} is a $G$-(weak) Hopf algebra $A$ with a $G$-equivariant $R$-matrix and a $G$-twist.
\end{Definition}

\begin{Lemma}\label{le:ribbon-ribbon}
A $G$-weak Hopf algebra $A$ can be endowed with the structure of a $G$-weak ribbon algebra if and only if the representation category $A\mod$ 
has the structure of a $G$-ribbon category.
\end{Lemma}

\begin{proof}
If $A$ is a $G$-ribbon algebra, it follows from the definition that $A\mod$ 
is a $G$-ribbon category. If on the other hand, $A\mod$ is a $G$-ribbon 
category with $G$-braiding $c$ and $G$-twist $\theta$, define an $R$-matrix 
and a twist of $A$ by
\begin{equation}
R = \tau\circ c_{A,A}(1_A\otimes 1_A)\quad\mbox{ and }\quad
\theta = \theta_A(1)^{-1} \,\,\, . 
\end{equation}
For $v\in V$, $w\in W$ let $\bar v:A \to V$, $\bar w:A \to W$ be the $A$-linear maps with $\bar v(1_A) = v, \bar w(1_A) = w$.
We then have
\begin{equation*}
 \tau((v\otimes w).R) =   \tau(\bar v \otimes \bar w(R)) = (\bar w \otimes \bar v) c_{A,A}(1_A\otimes 1_A)
 = c_{V,W}(v\otimes w)  \,\,\, , 
\end{equation*}
\begin{equation*}
v.\theta^{-1} =   v.(\theta_A(1_A)) = \bar v(\theta_A(1_A)) = \theta_V \bar v(1_A) = \theta_V(v).
\end{equation*}
Thus $R$ and $\theta$ satisfy the conditions of definition \ref{def:ribbon} by construction.
\end{proof}

As an immediate consequence of lemma \ref{le:ribbon-ribbon}, we have:

\begin{Kor}
If $A$ is a $G$-ribbon algebra, the strictification algebra $\As$ inherits the structure of a $G$-weak ribbon algebra such that the equivalence $F:A\mod \to \As\mod$ is an equivalence of $G$-ribbon categories.
\end{Kor}

\appendix
\section{Table summarizing terminology}
The following table summarizes the terminology
for Hopf algebras with an action of a finite group $G$
and their weakenings. We consider two types of weakenings: a weakening of the 
$G$-action corresponding to the two rows of
the table, and a weakening of the unitality of the coproduct,
corresponding to the two columns of the table.

Each square contains three different entries, depending
on additional structure on the Hopf algebra. The objects in 1.\
only have the $G$-action and no additional structure (see Definition 
\ref{def:Hopf-action}). The objects in 2.\ are equipped with a $G$-grading 
with the compatibilities introduced in Definition 
\ref{GHopf}. The objects in 3.\ have, in addition to the 
$G$-equivariant structure, a  $G$-equivariant $R$-matrix and a $G$-twist 
as introduced in Definition \ref{def:ribbon}.\\

\makebox[\linewidth]{
\begin{tabular}{l|c|c}
& Hopf algebra & weak Hopf algebra\\
\hline
strict $G$-action &
\begin{minipage}{5.5cm}
\begin{enumerate}
\item
Hopf algebra with strict $G$-action
\item
$G$-Hopf algebra with strict $G$-action
\item
$G$-ribbon algebra with strict $G$-action \phantom{g}
\end{enumerate}
\end{minipage}
 &
\begin{minipage}{5.5cm}
\begin{enumerate}
\item
weak Hopf algebra with strict $G$-action
\item
$G$-weak Hopf algebra with strict $G$-action
\item
$G$-weak ribbon algebra with strict $G$-action \phantom{g}
\end{enumerate}
\end{minipage}
\\
\hline
weak $G$-action & 
\begin{minipage}{5.5cm}
\begin{enumerate}
\item
Hopf algebra with weak $G$-action
\item
$G$-Hopf algebra
\item
$G$-ribbon algebra
\end{enumerate}
\end{minipage}
 &
\begin{minipage}{5.5cm}
\begin{enumerate}
\item
weak Hopf algebra with weak $G$-action
\item
$G$-weak Hopf algebra 
\item
$G$-weak ribbon algebra\phantom{g}
\end{enumerate}
\end{minipage}
\end{tabular}}\\


\bibliographystyle{alpha}
\bibliography{As}{}

\end{document}